\documentclass[11pt,reqno]{amsart}

\usepackage{cases}
\usepackage{times}
\usepackage{graphicx}
\usepackage{amssymb}
\usepackage{mathrsfs}
\usepackage{ulem}
\usepackage{xcolor}

\newtheorem{theorem}{Theorem}[section]

\newtheorem{proposition}[theorem]{Proposition}
\theoremstyle{definition}

\newtheorem{remark}[theorem]{Remark}

\numberwithin{equation}{section}
\newcommand{\beqm}{\begin{eqnarray*}}
\newcommand{\eqm}{\end{eqnarray*}}

\newcommand{\Sn}{\mathbb{S}_{n}}

\newcommand{\Cn}{\mathbb{C}^n}

\newcommand{\qp}{\mathcal{Q}_p}

\newcommand{\D}{\mathbb{D}}
\newcommand{\B}{\mathbb{B}_n}

\DeclareMathOperator{\dist}{dist}

\begin{document}

\title[Tent embedding for Besov spaces and integral operators]{Carleson measures, tent embeddings, and Volterra-type integral operators on the unit ball}

\author[X. Lv]{Xiaofen Lv}
\address{\noindent Xiaofen Lv \\Department of Mathematics\\ Huzhou Normal University, Huzhou
313000,   China } \email{lvxf@zjhu.edu.cn}

\author[X. Tang]{Xiaomin Tang}
\address{Xiaomin Tang \\Department of Mathematics\\ Huzhou Normal University, Huzhou
313000,   China } \email{txm@zjhu.edu.cn}

\author[J. A. Virtanen]{Jani A. Virtanen}
\address{Jani A. Virtanen, University of Eastern Finland, Joensuu, Finland\newline
\indent University of Reading, Reading, United Kingdom\newline \indent University of Helsinki, Helsinki, Finland}

\thanks {X. Lv is supported by  NNSF of China
(Grant No. 12431005, 12171150)  and  NNSF of
Zhejiang Province (Grant No. LMS26A010001). X. Tang is supported by
NNSF of Zhejiang Province  (Grant No.  LZ24A010004). J. Virtanen was supported in part by the Engineering and Physical Sciences Research Council (EPSRC) grant EP/Y008375/1 and the Research Council of Finland grant no.~359563.}



\begin{abstract}
In this paper, we establish a sharp comparison between Carleson-cube and Bergman-metric-ball conditions on the open unit ball $\B$ and combine it with a Berezin-type characterization to prove embedding theorems for Besov spaces and Bergman spaces on $\B$ into logarithmic tent spaces in the Bergman metric. As applications, we characterize the boundedness, compactness, and essential norms of the Volterra-type integral operators $T_g$ and $I_g$ acting from the Besov space $B_t(\B)$ to the general function space $F(p,q,s)$.
\end{abstract}

\maketitle

\section{Introduction}\label{Sect1}

Let  $dv$  denote the normalized volume measure on $\B$, and let  $H(\B)$ denote the set of all holomorphic functions on $\B$.  For $0<p<\infty$ and $\alpha>-1$, the weighted Bergman space  $A^p_\alpha(\B)$  consists of functions  $f\in H(\B)$  satisfying
$$
\|f\|_{A^p_\alpha}=\left(\int_{\B}|f(z)|^p dv_\alpha(z)\right)^{1/p}<\infty,
$$
where
$dv_\alpha(z)=c_\alpha(1-|z|^2)^\alpha dv(z)$, and
$c_\alpha$ is a positive normalizing constant such that $v_\alpha(\B)=1$.
 When $\alpha=0$, we simply write $A^p(\B)$ for $A^p_\alpha(\B)$.

The M\"{o}bius invariant measure $d\tau$ is defined by
$$
d\tau(z)=\frac{dv(z)}{(1-|z|^2)^{n+1}}.
$$
For $0< p<\infty$,  the Besov space $B_p(\B)$ is the space of all  $f\in H(\B)$ such that
\begin{equation}\label{Besov-1}
(1-|z|^2)^{N}\frac{\partial^m f}{\partial z^m}(z)\in L^p(\B, d\tau),\qquad |m|=N,
 \end{equation}
 for some positive integer $N>n/p$, where $m=(m_1, \cdots, m_n)$ is the multi-index, $|m|=m_1+ \cdots +m_n$.
 By \cite[Theorem 6.1]{Zh04}, $f\in B_p(\B)$  if and only if  \eqref{Besov-1} holds for every  positive integer $N>n/p$.
With the norm
$$
\|f\|_{B_p}=\sum_{|m|\leq N-1}\left|\frac{\partial^m f}{\partial z^m}(0)\right|+\sum_{|m|=N}\left(\int_{\B}\left|(1-|z|^2)^N\frac{\partial^m f}{\partial z^m}(z)\right|^pd\tau(z)\right)^{1/p}
$$
the space $B_p(\B)$ is a Banach space for $p\geq 1$. Furthermore, the polynomials are dense in  $B_p(\B)$. It is well known that
 $$
 B_{p_1}(\B)\subseteq B_{p_2}(\B)
 $$
 whenever $0<p_1\leq p_2\leq \infty$, where $B_\infty$ denotes the Bloch space $\mathcal{B}$, which plays an important role in the classical geometric function theory. Recall that, $f$ belongs to Bloch space $\mathcal{B}$ if  $f\in H(\B)$ and
 $$
 \sup\limits_{z\in \B}|\widetilde{\nabla}f(z)|<\infty,
 $$
where $\widetilde{\nabla}f(z)
=\nabla(f\circ\varphi_z)(0)$ and
$\nabla f=\left(\frac{\partial f}{\partial z_1}, \cdots, \frac{\partial f}{\partial z_n}\right)$
is the complex gradient of $f$.   When $n<p<\infty$, the Besov space $B_p(\B)$ can be equipped with the following norm
\begin{equation}\label{Besov-a1}
|f(0)|+\left(\int_{\B}(1-|z|^2)^p\left|\nabla f(z)\right|^pd\tau(z)\right)^{1/p}.
\end{equation}
Especially, on the unit disk $\D$, for $p>1$, $B_p(\D)$ consists of all $f\in H(\D)$ such that
$$
 \left(\int_{\D}(1-|z|^2)^{p-2}| f'(z)|^pdA(z)\right)^{1/p} < \infty,
$$
which was introduced by Zhu in \cite{Zhu91}. More information on Besov spaces can be found  in \cite[Chapter 6]{Zh04}.

Notice that each $B_p$  is  M\"{o}bius invariant for $0<p\le \infty$ and $B_1$ is the minimal  M\"{o}bius invariant space, which was first studied by Arazy~\cite{AF} on the unit disk and later by Peloso \cite{Pe92} on the unit ball. The space $B_2$ plays the role of the Dirichlet space. The uniqueness of $B_2$ among  M\"{o}bius invariant Hilbert spaces was first proved by  Arazy and Fisher~\cite{AF85} for the unit disk, and then generalized to the unit ball by Zhu in \cite{Zh91}. It is also known that the Bloch space is the maximal M\"{o}bius invariant Banach space.

In order to state our applications to integral operators, we also recall the general function space $F(p,q,s)$. For $0<p<\infty$, $-n-1<q<\infty$, $0\le s<\infty$, and $s+q>-1$, let $F(p,q,s)$ denote the space of all $f\in H(\B)$ such that
$$
\|f\|^p_{F(p,q,s)}
=
|f(0)|^p+\sup_{z\in\B}\int_{\B}|\nabla f(w)|^p(1-|w|^2)^qG^s(z,w)\,dv(w)<\infty.
$$
This space was introduced by Zhao on the unit disk in \cite{Zhao96} and has since been studied extensively on both the unit disk and the unit ball. The space $F(p,q,s)$ is often referred to as a general function space because it includes many classical function spaces as special cases for suitable choices of the parameters $p$, $q$, and $s$, such as $\mathrm{BMOA}$, $\mathcal Q_p$, Bergman spaces, Hardy spaces, Bloch spaces, and Dirichlet-type spaces; see \cite{PR06} for further details. For example, $F(2,2,s)$ is just the $\mathcal Q_s$ space; see~\cite{Oy98, Ya98}.

A central technical contribution of this paper is an equivalence between the Carleson-cube condition of Yang \cite{Ya98} and the Bergman-metric-ball condition on $\B$, established
here in a range that goes beyond the standard Bergman Carleson framework. In particular, Proposition~\ref{pro1} shows that for every $\gamma>1$,
$$
    \sup_{\zeta\in\Sn,\,r\in(0,1)}\frac{\mu(Q_r(\zeta))}{r^{n\gamma}}<\infty \iff \sup_{a\in\B}\frac{\mu(D(a,R))}{(1-|a|^2)^{n\gamma}}<\infty
$$
and that the same equivalence holds for the corresponding vanishing
conditions on the boundary. By contrast, Remark~\ref{remark2.4} shows that this equivalence fails in general for $(n-1)/n<\gamma\le 1$. Together, these results provide the key geometric tool that yields a Bergman-metric-ball characterization of $F(p,q,s)$ for $s>1$, which in turn underlies our analysis of the Volterra-type integral operators $T_g$ and $I_g$ in Section~\ref{section 3}.

For $g\in H(\B)$ and $z\in\B$, these integration operators are defined on $H(\B)$ by
$$
T_g f(z)=\int_0^1 f(tz)\,\Re g(tz)\,\frac{dt}{t}, \qquad
I_g f(z)=\int_0^1 g(tz)\,\Re f(tz)\,\frac{dt}{t},
$$
where $\Re g(z)=\sum_{j=1}^{n}z_j\,\partial g/\partial z_j(z)$ is the radial derivative of $g$. The operator $T_g$ is a generalization of the Ces\`{a}ro operator, and it was introduced on the unit disk by Pommerenke \cite{Po} and on the unit ball by Hu \cite{Hu}. More recently, Miihkinen, Pau, Per\"al\"a, and Wang \cite{MPPW} characterized the boundedness of $T_g$ from the weighted Bergman spaces $A^p_\alpha$ to the Hardy spaces $H^q$ on $\B$ for the full range $0<p,q<\infty$. We consider instead the boundedness, compactness, and essential norms of $T_g$ from Besov spaces to the general function space $F(p,q,s)$, and we also treat the companion operator $I_g$, which was not addressed in \cite{MPPW} and requires a separate analysis. The change of source and target spaces leads to a different approach: whereas \cite{MPPW} relies on Calder\'{o}n's area theorem and Coifman--Meyer--Stein cone tent spaces, our analysis proceeds via logarithmic tent spaces in the Bergman metric, the Carleson-cube/Bergman-ball comparison of Proposition~\ref{pro1}, and the Berezin-type characterization of \cite{Ya98}.

The remainder of the paper is organized as follows. Section~\ref{section 2} develops the Carleson-cube/Bergman-ball comparison and related results; Section~\ref{embedding} establishes the tent embedding theorems for Besov and Bergman spaces; and Section~\ref{section 3} contains the applications to the Volterra-type integral operators $T_g$ and $I_g$.

Throughout the paper, $A(f) \lesssim B(f)$ means that there exists a positive constant $C$, independent of $f$, such that $A(f)\leq CB(f)$. We write $A \asymp B$ if both $A \lesssim B$ and $B\lesssim A$.

\section{Characterizations of Carleson measures}\label{section 2}

Recall that the space $L^p(\B, d\mu)$  consists of all $\mu$-measurable functions $f$ on $\B$ satisfying
$$
\int_{\B}|f(z)|^p d\mu(z)<\infty,
$$
where $\mu$ is  a positive Borel measure on   $\B$ and $0<p<\infty$.

In this section we are interested in positive  measures $\mu$ with the property that $L^q(\B, d\mu)$ contains the Bergman  space. Such  measures are termed  Carleson measures, which play a significant role in the theory of function spaces. Carleson measures were first introduced by Carleson \cite{Car} to characterize the interpolating sequences in the algebra $H^\infty$
 of bounded analytic functions.
Theorem \ref{TH0} and Theorem \ref{TH0a} show that Carleson measures can be characterized by the Carleson cube or the Bergman metric ball;  Proposition \ref{pro1} and Remark~\ref{remark2.4} show the relationship between the Carleson cube and the Bergman metric ball.

For $\zeta\in\Sn=\partial\B$ and $r\in(0, 1)$, the Carleson cube $Q_r(\zeta)$ is defined by
\[Q_r(\zeta)=\left\{z\in\B: |1-\langle z, \zeta\rangle|<r\right\}.\]
It is well known that
$$v(Q_r(\zeta))\asymp r^{n+1}.$$
 The Bergman metric  on $\B$ is
$$
\beta(z, w)=\frac{1}{2}\log\frac{1+|\varphi_z(w)|}{1-|\varphi_z(w)|}, \qquad z, w\in\B,
$$
 where $\varphi_z$ is the involutive automorphism of $\B$ that interchanges $0$ and $z$.
For $R>0$ and $z\in\B$, define the Bergman metric ball to be
$$D(z, R)=\left\{w\in \B: \ \beta(z, w)<R \right\}.$$
 It is well known that,  for any fixed $R$
\begin{equation}\label{Bergman-1}
\left|1-\langle z, w\rangle\right|\asymp 1-|w|^2\asymp 1-|z|^2,   \qquad w\in D(z, R),
\end{equation}
and
$$
v(D(z, R))\asymp (1-|z|^2)^{n+1},\qquad z\in \B.
$$
 See more information in \cite{Zh04}.

For  $0<p, q<\infty$, $\alpha>-1$, a positive Borel measure $\mu$ on $\B$ is  called to be a $q/p$-Carleson measure for $A^p_\alpha (\B)$ if the embedding operator
$i:  A^p_\alpha (\B)\rightarrow L^q (\B, d\mu)$ is bounded. We will use $\|i\|_{A^p_\alpha(\B) \rightarrow L^q (\B, d\mu)}$ to denote
the operator norm. We say that $\mu$ is a vanishing $q/p$-Carleson measure for $A^p_\alpha (\B)$ if
$$\lim\limits_{j\rightarrow\infty}\int_{\B}\left|f_j(z)\right|^q d\mu(z)=0$$
whenever $\{f_j\}$ is  norm-bounded   in $ A^p_\alpha (\B)$ and
converges to $0$ uniformly on  any compact subset of $\B$ as
$j\to \infty$.

 We state the following result concerning   Carleson  measures, which have been studied by several authors,
including Hastings \cite{Ha75} and Luecking  \cite{Lu83} for $\D$;
Cima and Wogen \cite{CW}, Zhu \cite{Zh04} and Duren and Weir \cite{DW} for $\B$;
Zhu \cite{Zhu88} for bounded symmetric domains; Abate and Saracco \cite{AS}, Abate, Raissy
and Saracco \cite{ARS}, Hu, Lv and Zhu \cite{HLZ16}, Abate and Raissy \cite{AR} and Abate,  Mongodi and   Raissy \cite{AMR} for strongly pseudoconvex domains.

\begin{theorem}\label{TH0}
Suppose $0<p\leq q<\infty$, and $\mu$ is a positive Borel measure on $\B$.
Then  the following assertions are equivalent:
\begin{enumerate}
\item[(1)] $\mu$ is a $q/p$-Carleson measure for $A^p (\B)$;

\item[(2)] for some (or all) $R>0$,
$$\sup_{a\in\B}\frac{\mu(D(a, R))}
{(1-|a|^2)^{(n+1)q/p}}<\infty;$$

 \item[(3)] $\mu$ satisfies
\begin{equation}\label{car-tube1}
\sup_{\zeta\in \Sn, \, r\in(0,1)}\frac{\mu(Q_r(\zeta))}
{r^{(n+1)q/p}}<\infty.
\end{equation}
Moreover,
$$
\|i\|^{q}_{A^p(\B)\rightarrow L^{q}(\B, d\mu)}\asymp\sup_{a\in\B}\frac{\mu(D(a, R))}
{(1-|a|^2)^{(n+1)q/p}}
\asymp\sup_{\zeta\in \Sn, \, r\in(0,1)}\frac{\mu(Q_r(\zeta))}
{r^{(n+1)q/p}}.
$$
\end{enumerate}
\end{theorem}

\begin{proof}
Notice that the Bergman kernel in $A^2(\B)$ is $K(z,w)=\frac{1}{(1-\langle z, w\rangle)^{n+1}}$.
Taking $\alpha=0,\, t=2q/p$ in \cite[Theorem 3.1]{HLZ16}, $A^p(\B)$ is continuously contained in
$L^q(\B, d\mu)$ for $p\leq q$ if and only if condition (2) holds. Moreover, condition (2) is equivalent to
\begin{equation}\label{berez}
\sup_{z\in \B}\int_{\B}\frac{(1-|z|^2)^{(n+1)q/p}}{|1-\langle w, z\rangle|^{2(n+1)q/p}}d\mu(w)<\infty.
\end{equation}
 By \cite[Theorem 1]{Ya98}, we obtain the equivalence between \eqref{car-tube1} and \eqref{berez}.
\end{proof}

 Similar to the proof of Theorem \ref{TH0}, by \cite[Theorem 3.2]{HLZ16} and \cite[Lemma 1]{Hu03}, we obtain the following theorem.

\begin{theorem}\label{TH0a}
Suppose $0<p\leq q<\infty$, and $\mu$ is a positive Borel measure on $\B$.
Then  the following assertions are equivalent:
\begin{enumerate}
\item[(1)] $\mu$ is a vanishing $q/p$-Carleson measure for $A^p (\B)$;

\item[(2)] for some (or all) $R>0$,
$$
\lim_{|a|\rightarrow 1}\frac{\mu(D(a, R))}{(1-|a|^2)^{(n+1)q/p}}=0;
$$

\item[(3)] $\mu$ satisfies
$$\lim_{r\rightarrow 0} \frac{\mu(Q_r(\zeta))}
{r^{(n+1)q/p}}=0$$
 uniformly for $\zeta\in \Sn$.
 \end{enumerate}
\end{theorem}

Theorems~\ref{TH0} and~\ref{TH0a} show the close connection between Bergman balls and Carleson cubes for the standard Bergman
Carleson exponent $(n+1)q/p$, where the equivalence is classical. Proposition~\ref{pro1} below treats  the more delicate situation for the
exponent $n\gamma$ with $\gamma>1$.
Here the standard Bergman Carleson framework is unavailable: Indeed, the embedding $A^{p_0}\hookrightarrow L^{q_0}$ with
$q_0/p_0=n\gamma/(n+1)$ requires $q_0\ge p_0$, which forces $\gamma\ge (n+1)/n$ and therefore excludes the range $1<\gamma<(n+1)/n$. To cover the full range $\gamma>1$, we instead use an
endpoint $p=\infty$ extension of \cite[Lemma~2.5]{HLZ16}, combined with Yang's Berezin-type characterization \cite[Theorem~1]{Ya98} of
$p$-Carleson measures. Remark~\ref{remark2.4} below shows that the equivalence fails throughout the range $(n-1)/n<\gamma\le 1$, so in
particular the condition $\gamma>1$ cannot be weakened to $\gamma\ge 1$.

\begin{proposition}\label{pro1}
Suppose $\gamma>1$, and $\mu$ is a positive Borel measure on $\B$.
Then
\begin{enumerate}
\item[(1)]  for some (or all) $R>0$,
 \begin{equation}\label{Bergmanball-2}
\sup\limits_{a\in\B}\frac{\mu(D(a, R))}
{(1-|a|^2)^{n\gamma}}<\infty
\end{equation}
if and only if
 \begin{equation}\label{Bergmanball-3}
\sup\limits_{\zeta\in \Sn, \, r\in(0,1)}\frac{\mu(Q_r(\zeta))}
{r^{n\gamma}}<\infty;
\end{equation}

\item[(2)] for some (or all) $R>0$,
\begin{equation}\label{Bergmanball-4}
    \lim\limits_{|a|\rightarrow 1}\frac{\mu(D(a, R))}{(1-|a|^2)^{n\gamma}}=0
\end{equation}
if and only if
\begin{equation}\label{Bergmanball-5}
\lim\limits_{r\rightarrow 0}\frac{\mu(Q_r(\zeta))}
{r^{n\gamma}}=0
\end{equation}
uniformly for $\zeta\in \Sn$.
\end{enumerate}
\end{proposition}

\begin{proof}
We first show that the equivalence in \cite[Lemma 2.5]{HLZ16} extends to the
endpoint case $p =\infty$. The implication (a) \(\Rightarrow\) (b) follows from the pointwise estimate \cite[(2.10)]{HLZ16}. Conversely, suppose that (b) holds. By the proof of \cite[Lemma 2.5]{HLZ16},
\[
\widetilde{\mu}_{t}(a)\delta (a)^{s - (n + 1)(1 - t / 2)}\leq C T_{t, - s}(\widehat{\mu}_{R}\delta^{s})(a),
\]
and hence
\[
T_{t, - s}(\widehat{\mu}_{R}\delta^{s})(a)\leq \left\| \widehat{\mu}_{R}\delta^{s}\right\|_{L^\infty} T_{t, - s}\mathbf{1}(a),
\]
where
\[
T_{t, - s}\mathbf{1}(a) = \delta (a)^{(n + 1)(t - 1) + s}\int_{\mathbb{B}_{n}}|K(a,w)|^{t}\delta (w)^{-s}d v(w).
\]
By  \cite[(2.4)]{HLZ16}, we have
\[
\int_{\mathbb{B}_{n}}|K(a,w)|^{t}\delta (w)^{-s}d v(w)\lesssim \delta (a)^{-s + n + 1 - (n + 1)t},
\]
since \(t > \frac{n}{n + 1}\) and \((n + 1)(1 - t)< s< 1\). Therefore
\[
T_{t, - s}\mathbf{1}(a)\lesssim 1,
\]
so \(T_{t, - s}\mathbf{1}\in L^{\infty}\). It follows that
\[
\widetilde{\mu}_{t}(a) \delta(a)^{s - (n + 1)(1 - t / 2)}\lesssim \| \widehat{\mu}_{R}\delta^{s}\|_{L^\infty},
\]
and hence (a) holds.  Similarly,  (b) and (c) in  \cite[Lemma 2.5]{HLZ16} are also equivalent when \(p = \infty\).

We now prove the equivalence of \eqref{Bergmanball-2} and \eqref{Bergmanball-3}. On the unit ball, notice that $\delta(z)=1-|z|$ and $D(a, R)$ is the Bergman metric ball. Then
\[
\widehat{\mu}_{R}(a) = \frac{\mu(D(a, R))}{v(D(a,R))}\asymp \frac{\mu(D(a, R))}{(1-|a|)^{n+1}}
\]
is exactly the averaging quantity appearing in \cite[Lemma 2.5]{HLZ16}. Set \(t = \frac{2n\gamma}{n + 1}\) and \(s = n + 1 - n\gamma\). Since \(\gamma >1\), we have \(t > \frac{n}{n + 1}\) and \((n + 1)(1 - t)< s< 1\).
 That lemma shows that, for \(1< p\leq \infty\)
\[
\widetilde{\mu}_{t}(a)(1 - |a|)^{s - (n + 1)(1 - t / 2)}\qquad \text{and}\qquad \widehat{\mu}_{R}(a)(1 - |a|)^{s}
\]
have comparable \(L^{p}(\mathbb{B}_{n})\) norms. Especially,
\[
\widehat{\mu}_{R}(a)(1 - |a|)^{s}\in L^{\infty}(\mathbb{B}_{n})\iff \widetilde{\mu}_{t}(a)(1 - |a|)^{s - (n + 1)(1 - t / 2)}\in L^{\infty}(\mathbb{B}_{n}).
\]
Since \(s - (n + 1)(1 - \frac{t}{2}) = 0\) and \(1 - |a|\asymp 1 - |a|^{2}\), it follows that \eqref{Bergmanball-2} is equivalent to
\[
\sup_{a\in \mathbb{B}_{n}}\widetilde{\mu}_{t}(a)< \infty.
\]
Notice that
\[
\widetilde{\mu}_{t}(a) = \int_{\mathbb{B}_{n}}\left(\frac{|K(a,z)|}{\sqrt{K(a,a)}}\right)^{t}d\mu (z)
= \int_{\mathbb{B}_{n}}\frac{(1 - |a|^{2})^{n\gamma}}{|1 - \langle z,a\rangle|^{2n\gamma}} d\mu (z),
\]
where
\[
K(z,w) = \frac{1}{(1 - \langle z,w\rangle)^{n + 1}}
\]
is the Bergman kernel. Therefore \eqref{Bergmanball-2} is equivalent to
\[
\sup_{a\in \mathbb{B}_{n}}\int_{\mathbb{B}_{n}}\frac{(1 - |a|^{2})^{n\gamma}}{|1 - \langle z,a\rangle|^{2n\gamma}} d\mu (z)< \infty,
\]
which in turn is equivalent to   \eqref{Bergmanball-3} by \cite[Theorem~1]{Ya98}.

Moreover, the proof of part~(1) yields the estimate
\begin{equation}\label{Bergmanball-6}
\sup_{\zeta\in\Sn,\ \tau\in(0,1)}
\frac{\nu(Q_\tau(\zeta))}{\tau^{n\gamma}}
\lesssim
\sup_{a\in\B}\frac{\nu(D(a,R))}{(1-|a|^2)^{n\gamma}}
\end{equation}
for every positive Borel measure $\nu$ on $\B$, with an implicit constant
independent of $\nu$.

It remains to prove the second assertion. Suppose that~\eqref{Bergmanball-4} holds. For $0<r<1$, define
$$
d\mu_r=\chi_{\{|z|>r\}}\,d\mu.
$$
We show that
\begin{equation}\label{claim}
    \lim_{r\to 1} \sup_{a\in\B}\frac{\mu_r(D(a,R))}{(1-|a|^2)^{n\gamma}}=0.
\end{equation}
We may assume that there is a $w$ in $D(a,R)\cap\{|w|>r\}$. Then \eqref{Bergman-1} implies that $1-|w|^2\asymp 1-|a|^2$. As $|w|>r$, we have $1-|a|^2\asymp 1-|w|^2<1-r^2\to0$ as $r\to1$, so $|a|\to1$. Since $\mu_r(D(a,R))\le \mu(D(a,R))$, the claim in~\eqref{claim} follows from
\eqref{Bergmanball-4}.

Applying \eqref{Bergmanball-6} with $\nu=\mu_r$ and using \eqref{claim}, we obtain
$$
\sup_{\zeta\in\Sn,\ \tau\in(0,1)}
\frac{\mu_r(Q_\tau(\zeta))}{\tau^{n\gamma}}\to0
\qquad\text{as }r\to1.
$$
Now, if $z\in Q_\tau(\zeta)$, then $1-|z|
\le 1-|\langle z,\zeta\rangle|
\le |1-\langle z,\zeta\rangle|
<\tau$.
Hence $Q_\tau(\zeta)\subset\{|z|>r\}$ whenever $\tau<1-r$, and therefore
$$
\mu(Q_\tau(\zeta))=\mu_r(Q_\tau(\zeta))
\qquad\text{for }0<\tau<1-r.
$$
It follows that
$$
\sup_{\zeta\in\Sn,\ 0<\tau<1-r}
\frac{\mu(Q_\tau(\zeta))}{\tau^{n\gamma}}\to0
\qquad\text{as }r\to1^-,
$$
which is \eqref{Bergmanball-5}.

Conversely, assume that \eqref{Bergmanball-5} holds. Let $a\in\B\setminus\{0\}$ and set $\zeta_a=a/|a|$.
If $z\in D(a,R)$, then \eqref{Bergman-1} implies
that $|1-\langle z,a\rangle|\lesssim 1-|a|^2$. Hence,
\begin{equation*}
    |1-\langle z,\zeta_a\rangle| \le |1-\langle z,a\rangle|+|\langle z,a-\zeta_a\rangle| \lesssim 1-|a|^2+|a-\zeta_a|\lesssim 1-|a|^2.
\end{equation*}
Therefore, $D(a,R)\subset Q_{C(1-|a|^2)}(\zeta_a)$ for some $C>0$ independent of $a$, and so
$$
\frac{\mu(D(a, R))}{(1-|a|^2)^{n\gamma}}
\lesssim
\frac{\mu(Q_{C(1-|a|^2)}(\zeta_a))}{(C(1-|a|^2))^{n\gamma}}.
$$
Since $C(1-|a|^2)\to0$ as $|a|\to1$, \eqref{Bergmanball-5} implies \eqref{Bergmanball-4}.
\end{proof}

\begin{remark}\label{remark2.4}  Proposition \ref{pro1} is invalid for $(n-1)/n< \gamma\leq 1$.
\end{remark}

In order to explain this clearly, we need to introduce the  $\qp$ space.
The
invariant Green's function on $\B$ is given by
$$
G(z, a)=g\left(\varphi_a(z)\right),
$$
 where
$$g(z)=\frac{n+1}{2n}\int_{|z|}^{1}\frac{(1-t^2)^{n-1}}{t^{2n-1}}\,dt.$$ The $\qp$ space  consists of all   $f\in H(\B)$
such that
$$
\|f\|_{\qp}=|f(0)|+\sup_{a\in \B}\left(\int_{\B}|\widetilde{\nabla}f(z)|^2G^p(z,a)
\,d\tau(z)\right)^{1/2}<\infty.
$$
When $0<p\leq(n-1)/n$ or $p\geq n/(n-1)$, the space $\qp$
contains only the constant functions; when $(n-1)/n<p<n/(n-1)$, $\qp$  space  contains all polynomials.  Moreover, $\mathcal{Q}_1=\mathrm{BMOA}$, and
 $\mathcal{Q}_p=\mathcal{B}$ whenever $1<p<n/(n-1)$.  Note that $$\mathcal{Q}_p\subsetneq \mathrm{BMOA}\subsetneq\mathcal{B} \qquad \textrm{if} \quad (n-1)/n<p<1.$$ More information can be found in \cite{Xiao01, Xiao06, Xiao08} on the unit disk or \cite{Oy98, Ya98} on the unit ball.

To verify Remark~\ref{remark2.4}, choose some function $f\in \mathcal{B}\backslash \mathcal{Q}_\gamma$ for $(n-1)/n<\gamma\leq 1$. Setting
$$d\nu(z)=(1-|z|^2)^{n\gamma}|\widetilde{\nabla} f(z)|^2\,d\tau(z),$$
then $\nu$
is a positive Borel measure on $\B$.
Since $f$ does not belong to $\mathcal{Q}_\gamma$ space, by \cite[Theorem 2]{Ya98} (the semi-norm for  $\mathcal{Q}_\gamma$  is comparably dominated by the geometric quantity), we obtain
$$ \sup_{\zeta\in \Sn,\, r\in(0, 1)}\frac{\nu(Q_r(\zeta))}
{r^{n\gamma}}=\infty.$$
On the other hand, for $R>0$, we obtain
\begin{align*}
\sup_{a\in \B}\frac{\nu(D(a, R))}{(1-|a|^2)^{n\gamma}}&=\sup_{a\in \B}
\frac{\displaystyle\int_{D(a, R)}(1-|z|^2)^{n\gamma}|\widetilde{\nabla} f(z)|^2\,d\tau(z)}{(1-|a|^2)^{n\gamma}}\\
&\lesssim\|f\|^2_{\mathcal{B}}\sup_{a\in \B}\displaystyle\int_{D(a, R)}d\tau(z)
\lesssim\|f\|^2_{\mathcal{B}}<\infty.
\end{align*}

\section{Tent embedding theorems}\label{embedding}
In this section we will  study  embedding theorems. First, we will introduce  some notations.
Suppose  $0<q<\infty$ and $s, t$ are real numbers. Let $\mu$ be a positive Borel measure  on $\B$. For $R>0$, the   space  $T^{q}_{s, t, R}(\B, \mu)$ is the set of all $\mu$-measurable functions $f$ on
$\B$ such that
$$\|f\|_{T^{q}_{s, t, R}(\B, \mu)}^q=\sup_{a\in\B}\frac{\int_{D(a, R)}|f(z)|^q\,d\mu(z)}{(1-|a|^2)^s\left(\log\frac{2}
{1-|a|} \right)^t} <\infty.$$
By the standard   covering argument, we know that the space $T^q_{s,t, R}(\B, \mu)$ is independent of the
radius $R$, denoted by  $T^q_{s,t}(\B, \mu)$ for simplicity. Borrowing terminology from harmonic
analysis, we will call $T^q_{s,t}(\B, \mu)$ a tent space, or more precisely, a logarithmic
tent space in the Bergman metric.

We now present the embedding theorem from Besov spaces to logarithmic tent spaces in the Bergman metric, stated as follows.

\begin{theorem}\label{TH2}
Let $1\leq p<\infty$, $0<q<\infty$ and let $s, t$ be  real numbers. Suppose $\mu$ is a positive Borel measure on $\B$. Then   $i: B_p(\B)\rightarrow T_{s, t}^{q}(\B, \mu)$ is bounded if and only if
$$\|\mu\|_{s,\, t,\, p,\, q}:=\sup_{a\in\B}\frac{\mu(D(a, R))}{(1-|a|^2)^s\left(\log\frac{2}{1-| a|^2}\right)^{t-q+\frac{q}{p}}}<\infty$$
for some (or any) $R>0$. Furthermore,
\begin{equation}\label{Eq2.02}
\|i\|^q_{B_p(\B)\rightarrow T_{s, t}^{q}(\B, \mu)}\asymp\|\mu\|_{s,\, t,\, p,\, q}.
\end{equation}
\end{theorem}

\begin{proof}
To prove boundedness, we construct a test function belonging to the Besov space $B_p(\B)$. For $a=(a_1, \ldots, a_n)\in \B$,  consider
$$
    f_a(z)=\log\frac{2}{1-\langle z, a\rangle}.
$$
Then
$$
 \frac{\partial f_a}{\partial z_j}(z)=\frac{\overline{a_j}}{1-\langle z, a\rangle},\qquad \frac{\partial^2 f_a}{\partial z_j\partial z_k}(z)=\frac{\overline{a_ja_k}}{(1-\langle z, a\rangle)^2}, \qquad \ldots,
$$
similarly,  for any multi-index $m$ with $|m|=N$,
$$
    \left|\frac{\partial^m f_a}{\partial z^m}(z)\right|\lesssim\frac{1}{\left|1-\langle z, a\rangle\right|^N}
$$
Thus, when $N>n/p$, \cite[Theorem 1.12]{Zh04} implies
\begin{align*}
    \|f_a\|_{B_p}&\lesssim 1+\sum_{|m|=N}\left(\int_{\B}\frac{(1-|z|^2)^{Np-(n+1)}}{\left|1-\langle z, a\rangle\right|^{Np}}dv(z)\right)^{1/p}\\
    &\lesssim \left(\log\frac{2}{1-| a|^2}\right)^{1/p}.
\end{align*}
Notice that, for $z, w\in \B$, $\left|\log \frac{2}{z}\right|\asymp \left|\log \frac{2}{w}\right|$ whenever $|z|\asymp |w|$ and  $|w|\rightarrow 0$, thus $$|f_a(z)|\asymp\log\frac{2}{1-|a|^2},\qquad z\in D(a, R)$$ if $|a|$ is sufficiently close to $1^-$.
Since $i: B_p(\B) \to T_{s, t}^{q}(\B, \mu)$ is bounded, we obtain
\begin{align*}
\frac{\mu\left(D(a, R)\right)}{(1-|a|^2)^{s}\left(\log\frac{2}{1-| a|^2}\right)^{t-q}}&\lesssim\frac{\int_{D(a, R)}|f_a(z)|^q \,d\mu(z)}{(1-|a|^2)^{s}
\left(\log\frac{2}{1-|a|^2}\right)^t}\\
&\lesssim\|i\|^q_{B_p(\B)\rightarrow T_{s, t}^{q}(\B, \mu)}\|f_a\|^q_{B_p}\\
&\lesssim\|i\|^q_{B_p(\B)\rightarrow T_{s, t}^{q}(\B, \mu)}\left(\log\frac{2}{1-| a|^2}\right)^{q/p}.
\end{align*}
 This shows that
\begin{equation}\label{oEq01}
\sup_{a\in \B}\frac{\mu\left(D(a, R)\right)}{(1-|a|^2)^{s}\left(\log\frac{2}{1-| a|^2}\right)^{t-q+\frac{q}{p}}}\lesssim\|i\|^q_{B_p(\B)\rightarrow T_{s, t}^{q}(\B, \mu)}<\infty.
\end{equation}

Conversely, for any $f\in B_p(\B)$ and $a\in\B$, we need to prove
\begin{equation}\label{oEq02e2}
\frac{\int_{D(a, R)}
|f(z)|^q\,d\mu(z)}{(1-|a|^2)^{s}\left(\log\frac{2}{1-|a|^2}\right)^t}\lesssim \|\mu\|_{s,\, t,\, p,\, q}\|f\|_{B_p}^q.
\end{equation}
In fact, for $f\in B_p(\B)$, by \cite[Page 232 Exercise 6.21]{Zh04}, we have
\begin{equation}\label{besoveq}
|f(z)-f(a)|\lesssim\beta(z, a)^{1-\frac{1}{p}}\|f\|_{B_p},\qquad z, a\in\B.
\end{equation}
Set $z=0$. Then
$$|f(a)|\lesssim\left(\log\frac{2}{1-|a|^2}\right)^{1-\frac{1}{p}}\|f\|_{B_p}.$$
Thus
\begin{equation}\label{eq0l1}
\frac{\int_{D(a, R)}
|f(a)|^q\,d\mu(z)}{(1-|a|^2)^{s}\left(\log\frac{2}{1-|a|^2}\right)^t}\lesssim\frac{\mu\left(D(a, R)\right)}{(1-|a|^2)^{s}\left(\log\frac{2}{1-|a|^2}\right)^{t-q+\frac{q}{p}}}\|f\|^q_{B_p}.
\end{equation}
Next, we will  prove
\begin{equation}\label{Eq2.3}
\frac{\int_{D(a, R)}
|f(z)-f(a)|^q\,d\mu(z)}{(1-|a|^2)^{s}\left(\log\frac{2}{1-|a|^2}\right)^t}\lesssim\|\mu\|_{s,\, t,\, p,\, q}\|f\|^q_{B_p}.
\end{equation}
Notice that $q-q/p\geq 0$ and $$\log[2/(1-|z|^2)]\asymp \log[2/(1-|a|^2)]\ge\log2, \qquad z\in D(a, R)$$ if $|a|$ is sufficiently close to $1^-$.
From (\ref{Bergman-1}), we obtain
\begin{align*}
I(f)&:=\frac{\int_{D(a, R)}
|f(z)-f(a)|^q\,d\mu(z)}{(1-|a|^2)^{s}\left(\log\frac{2}{1-|a|^2}\right)^t}\\
&\lesssim (1-|a|^2)
\int_{D(a, R)}\frac{|f(z)-f(a)|^q }{\left|1-\langle z, a\rangle
\right|^{2n+3}}(1-|z|^2)^{2(n+1)-s}\\
&\qquad\qquad \times \left(\log\frac{2}{1-|z|^2}\right)^{q-q/p-t}\,d\mu(z)
\\
&=(1-|a|^2)\int_{D(a, R)}\frac{|f(z)-f(a)|^q }{\left|1-\langle z, a\rangle
\right|^{2n+3}}\,d\lambda(z)\\
&\le(1-|a|^2)\int_{\B}\frac{|f(z)-f(a)|^q }{\left|1-\langle z, a\rangle
\right|^{2n+3}}\,d\lambda(z),
\end{align*}
where $\lambda$ is a positive Borel measure on $\B$ defined to be $$
d\lambda(z)=(1-|z|^2)^{2(n+1)-s}\left(\log\frac{2}{1-|z|^2}\right)^{q-q/p-t}\,d\mu(z).
$$
By Theorem \ref{TH0} and (\ref{Bergman-1}), the condition $\|\mu\|_{s,\, t,\, p,\, q}<\infty$
implies that the embedding operator
$$i: A^{q/2}(\B)\rightarrow L^{q}(\B, d\lambda)$$
is bounded, furthermore
$$
\|\mu\|_{s,\, t,\, p,\, q}\asymp\sup_{a\in\B}\frac{\lambda\left(D(a, R)\right)}{(1-|a|^2)^{2(n+1)}}
\asymp\|i\|^q_{A^{q/2}(\B)\rightarrow L^{q}(\B, d\lambda)}.
$$
Thus,
\begin{align*}
I(f)\lesssim\|\mu\|_{s,\, t,\, p,\, q}(1-|a|^2)\left(\int_{\B}\frac{|f(z)-f(a)|^{q/2}}
{\left|1-\langle z,
a\rangle\right|^{\frac{2n+3}{2}}}\,dv(z)\right)^{2}.
\end{align*}
This, together with \eqref{besoveq}, gives
\begin{align*}
I(f)\lesssim\|\mu\|_{s,\, t,\, p,\, q}\|f\|^q_{B_p}(1-|a|^2)\left(\int_{\B}\frac{\beta(z,a)^{\frac{q-q/p}{2}}}{\left|1-\langle z, a\rangle
\right|^{\frac{2n+3}{2}}}\,dv(z)\right)^{2}.
\end{align*}
By \cite[Lemma 1]{LZ19}, we obtain
$$\int_{\B}\frac{\beta(z,a)^{\frac{q-q/p}{2}}}{\left|1-\langle z, a\rangle
\right|^{\frac{2n+3}{2}}}\,dv(z)\lesssim\frac{1}{(1-|a|^2)^{1/2}}.$$
This proves \eqref{Eq2.3}, with \eqref{eq0l1},  yields \eqref{oEq02e2}.
Therefore,  the embedding operator $i$ is bounded from $B_p(\B)$ to $T_{s, t}^{q}(\B, \mu)$. Estimate \eqref{Eq2.02} follows from \eqref{oEq01} and  \eqref{oEq02e2}.
\end{proof}

A bounded linear operator $T: B_p(\B)\to T^q_{s, t}(\B, \mu)$ is said to be compact if
$\|Tf_k\|_{T^q_{s, t}(\B, \mu)}\to0$ for every bounded sequence $\{f_k\}$ in $B_p(\B)$
that converges to $0$ uniformly on compact subsets of $\B$. The following result
is a natural companion to Theorem~\ref{TH2}.

\begin{theorem}\label{TH3}
Let $1\leq p<\infty$, $0<q<\infty$ and let  $s, t$ be real numbers. Suppose $\mu$ is a positive Borel measure on $\B$. Then  the embedding operator $i: B_p(\B)\rightarrow T_{s, t}^{q}(\B, \mu)$ is compact if and only if
\begin{equation}\label{log-car}
\lim_{|a|\rightarrow 1}\frac{\mu(D(a, R))}{(1-|a|^2)^{s}\left(\log\frac{2}{1-| a|^2}\right)^{t-q+\frac{q}{p}}}=0
\end{equation}
for some (or all) $R>0$.
\end{theorem}

\begin{proof}
The proof is similar to that of Theorem \ref{TH2}, with the following adjustments.
First, for any $a\in \B$, set
$$
f_a(z)=\left(\log\frac{2}{1-|a|^2}\right)^{-1/p}\log\frac{2}{1-\langle z, a\rangle}, \qquad z\in\B.
$$
By straight calculation as in the proof  of Theorem \ref{TH2}, we know
 $\left\{f_a\right\}_{a\in \B}$ is norm-bounded in $B_p(\B)$ and
$\left\{f_a\right\}$ converges to $0$ uniformly on any compact subset of
$\B$ as $|a|\rightarrow 1$. Suppose   $i$ is compact from $B_p(\B)$ to $T_{s, t}^{q}(\B, \mu)$, then
$$\frac{\mu(D(a, R))}{(1-|a|^2)^{s}\left(\log\frac{2}{1-| a|^2}\right)^{t-q+\frac{q}{p}}}\lesssim\frac{\int_{D(a, R)}|f_a(z)|^q\,d\mu(z)}{(1-|a|^2)^{s}
\left(\log\frac{2}{1-|a|^2}\right)^t}\to0$$
as $|a|\rightarrow 1$.

Conversely, suppose condition \eqref{log-car} holds. For any $r\in(0,1)$, define the cut-off measure
$$d\mu_r=\chi_{\{w\in \B: |w|>r\}}\,d\mu.$$
 By  \eqref{log-car}, we  conclude that $\|\mu\|_{s,\, t,\, p,\, q}<\infty$ and
$$\|\mu_r\|_{s,\, t,\, p,\, q}=\sup_{a\in\B}\frac{\mu_r\left(D(a, R)\right)}
{(1-|a|^2)^{s}\left(\log\frac{2}{1-| a|^2}\right)^{t-q+\frac{q}{p}}}\rightarrow  0$$
as $r \rightarrow 1$. Let $\{f_j\}$ be any norm-bounded sequence in $B_p(\B)$ with $f_j\rightarrow 0$ uniformly on any
compact subset of $\B$. Since $$
\left(\log\frac{2}{1-|a|^2}\right)^{q/p-q}\lesssim 1,
$$
for any $\varepsilon>0$, if  $r$ is sufficiently close to $1$, we have
\begin{align*}
\sup_{a\in \B}&\frac{\int_{D(a, R)}|f_j(z)|^q\,d\mu(z)}{(1-|a|^2)^{s}\left(\log\frac{2}{1-|a|^2}\right)^t}
\\
&\leq \sup_{a\in \B}\frac{\int_{D(a, R)\cap \{|z|\leq r\}}|f_j(z)|^q\,d\mu(z)
+\int_{D(a, R)}|f_j(z)|^q\,d\mu_r(z)}{(1-|a|^2)^{s}\left(\log\frac{2}{1-|a|^2}\right)^t}
\\
&\lesssim\|\mu\|_{s,\, t,\, p,\, q}\sup_{|z|\leq r}|f_j(z)|^q
+\|f_j\|^q_{B_p}\|\mu_r\|_{s,\, t,\, p,\, q}<\varepsilon
\end{align*}
whenever $j$ large enough. Therefore,
$$\lim_{j\rightarrow\infty}\|f_j\|_{T_{s, t}^{q}(\B, \mu)}=0,$$
which shows that $i: B_p(\B)\rightarrow T_{s, t}^{q}(\B, \mu)$ is compact.
\end{proof}

We are also interested in positive Borel measures $\mu$ on $\B$ with the property that $T_{s, t}^{q}(\B, \mu)$ contains the Bergman space $A^p_\alpha(\B)$. Similar to the proof above, we can obtain the following theorem.

\begin{theorem}\label{TH4}
Let $0<p,q<\infty$, $\alpha>-1$ and let $s, t$ be  real numbers. Suppose $\mu$ is a positive Borel measure on $\B$. Then
\begin{enumerate}
\item[(1)]
the embedding operator  $i: A^p_\alpha(\B)\rightarrow T_{s, t}^{q}(\B, \mu)$ is bounded if and only if
$$
\sup_{a\in\B}\frac{\mu(D(a, R))}{(1-|a|^2)^{s+(n+1+\alpha)q/p}\left(\log\frac{2}{1-| a|^2}\right)^{t}}<\infty
$$
for some (or any) $R>0$. Furthermore,
$$
\|i\|^q_{A^p_\alpha(\B)\rightarrow T_{s, t}^{q}(\B, \mu)}\asymp \sup_{a\in\B}\frac{\mu(D(a, R))}{(1-|a|^2)^{s+(n+1+\alpha)q/p}\left(\log\frac{2}{1-| a|^2}\right)^{t}}.
$$

\item[(2)] the embedding operator  $i: A^p_\alpha(\B)\rightarrow T_{s, t}^{q}(\B, \mu)$ is compact if and only if
$$\lim_{|a|\rightarrow 1}\frac{\mu(D(a, R))}{(1-|a|^2)^{s+(n+1+\alpha)q/p}\left(\log\frac{2}{1-| a|^2}\right)^{t}}=0$$
for some (or any) $R>0$.
\end{enumerate}
\end{theorem}

\begin{proof}
For (1), set
\begin{equation}\label{fun-test1}
f_a(z)=\frac{(1-|a|^2)^{\frac{n+1+\alpha}{p}}}{(1-\langle z,a \rangle)^{\frac{2(n+1+\alpha)}{p}}},\qquad z\in\B
\end{equation}
for any fixed $a\in\B$. Then $\|f_a\|_{A^p_\alpha}\lesssim 1$.
If $i: A^p_\alpha(\B)\rightarrow T_{s, t}^{q}(\B, \mu)$ is bounded, then
\begin{align*}
\frac{\mu(D(a, R))}{(1-|a|^2)^{s+(n+1+\alpha)q/p}\left(\log\frac{2}{1-| a|^2}\right)^{t}}&\asymp\frac{\int_{D(a, R)}|f_a(z)|^q \,d\mu(z)}{(1-|a|^2)^{s}
\left(\log\frac{2}{1-|a|^2}\right)^t}\\
&\lesssim\|i\|^q_{A^p_\alpha(\B)\rightarrow T_{s, t}^{q}(\B, \mu)}\|f_a\|^q_{A^p_\alpha}\\
&\lesssim\|i\|^q_{A^p_\alpha(\B)\rightarrow T_{s, t}^{q}(\B, \mu)}.
\end{align*}
Conversely, for any function $f\in A^p_\alpha(\B)$, by the  mean-value-inequality of subharmonic functions, we have
\begin{align*}
\frac{\int_{D(a, R)}|f(z)|^q \,d\mu(z)}{(1-|a|^2)^{s}
\left(\log\frac{2}{1-|a|^2}\right)^t}&\leq \frac{\mu(D(a, R))\sup\limits_{z\in D(a, R)}|f(z)|^q }{(1-|a|^2)^{s}
\left(\log\frac{2}{1-|a|^2}\right)^t}
\\
&\lesssim\frac{\mu(D(a, R))}{(1-|a|^2)^{s+(n+1+\alpha)q/p}\left(\log\frac{2}{1-| a|^2}\right)^{t}}\|f\|^q_{A^p_\alpha}.
\end{align*}

For (2), the proof is similar to that of (1), with minor modifications. Notice that  $\{f_a\}_{a\in \B}$ as in \eqref{fun-test1} is a bounded-norm sequence in  $A^p_\alpha(\B)$ and $f_a\rightarrow 0$ uniformly on compact subsets of $\B$ as $|a|\rightarrow 1$.
\end{proof}

\section{Applications to integral operators on  Besov spaces}\label{section 3}

In this section, we apply the tent embedding theorems established in Section~\ref{embedding} to study two integral operators. Recall from Section~\ref{Sect1} that, for $g \in H(\B)$, the Volterra-type integral operators $T_g$ and $I_g$ are defined on $H(\B)$ by
$$
    T_g f(z)=\int_0^1 f(tz)\,\Re g(tz)\,\frac{dt}{t}, \qquad
    I_g f(z)=\int_0^1 g(tz)\,\Re f(tz)\,\frac{dt}{t},
$$
where
$$
    \Re g(z)=\sum_{j=1}^{n}z_j\frac{\partial g}{\partial z_j}(z)
$$
is the radial derivative of $g$. Since
$$
    f(z)-f(0)=\int_0^1\frac{\Re f(tz)}{t}\,dt\quad \textrm{ and } \quad  \Re (fg)=g\Re f+f\Re g,
$$
it follows that
$$
    T_g f + I_g f + f(0)g(0) = gf,
$$
which shows that $T_g$ and $I_g$ are closely related to the multiplication operator  $M_g$,  defined by $M_g f = fg$. Further information on these integral operators can be found in \cite{JP24, MPPW, Xiao04, YL22} and the references therein.

As in the proof of \cite[Theorem 2]{Ya98},  $\|f\|^p_{F(p,\,q,\,s)}$ is equivalent to
$$
|f(0)|^p+\sup_{z\in \B}\int_{\B}|\nabla f(w)|^p(1-|w|^2)^q(1-|\varphi_z(w)|^2)^{ns}dv(w),
$$
which in turn is equivalent to
$$
|f(0)|^p+\sup_{z\in \B}\int_{\B}|\Re f(w)|^p(1-|w|^2)^q(1-|\varphi_z(w)|^2)^{ns}dv(w),
$$
see \cite[Theorem 3.1]{ZHC13}. Hence, \cite[Theorem 1]{Ya98} implies that $f\in F(p,\,q,\,s)$   if and only if
$$\sup_{\zeta\in \Sn, \, r\in(0,1)}\frac{\nu_f(Q_r(\zeta))}
{r^{ns}}<\infty,$$  where
 $d\nu_f(w)=|\Re f(w)|^p(1-|w|^2)^{q +ns}dv(w)$.
For $s>1$,   Proposition  \ref{pro1} shows that $f\in F(p,\,q,\,s)$  if and only if
$$
\sup_{a\in \B}\frac{\int_{D(a, R)}|\Re f(w)|^p(1-|w|^2)^{q +ns}dv(w)}{(1-|a|^2)^{ns}}<\infty
$$
for some (or all) $ R>0$. Moreover,
\begin{equation}\label{car-yub}
\|f\|^p_{F(p,\,q,\,s)}\asymp |f(0)|^p+\sup_{a\in \B}(1-|a|^2)^{q}\int_{D(a, R)}|\Re f(w)|^pdv(w).
\end{equation}

To characterize the boundedness, compactness, and essential norms of $T_g$ and $I_g$ acting from the Besov space $B_t(\B)$ to  $F(p,\,q,\,s)$, we introduce  Bloch-type and Herz-type spaces. For  $\alpha>0$ and $\beta \in \mathbb R$, the Bloch-type space $\mathcal{B}_{\alpha,\,\beta}$ consists of all $g\in H(\B)$ such that
$$
\|g\|_{\mathcal{B}_{\alpha, \beta}}=\sup_{z\in\B}|\Re g(z)|(1-|z|^2)^{\alpha}\left(\log\frac{2}{1-| z|^2}\right)^{\beta}<\infty,
$$
and the little Bloch-type space $\mathcal{B}^0_{\alpha,\,\beta}$ is the subspace of $\mathcal{B}_{\alpha,\,\beta}$ consisting of those $g \in \mathcal{B}_{\alpha,\,\beta}$ such that
$$
\lim_{|z|\rightarrow 1}|\Re g(z)|(1-|z|^2)^\alpha\left(\log\frac{2}{1-| z|^2}\right)^{\beta}=0.
$$
When $\alpha=1$ and $\beta=0$, $\mathcal{B}_{\alpha,\,\beta}$ and $\mathcal{B}^0_{\alpha,\,\beta}$ are  the classical Bloch space and the little Bloch  space, see  \cite{Zh04} for more information.

Recall that, for $\alpha\in \mathbb R$, the weighted Herz space, denoted by $\mathcal{H}^\infty_{\alpha}$, consists of all $f\in H(\B)$ for which
$$
\|f\|_{\mathcal{H}^\infty_{\alpha}}=\sup_{z\in \B} (1-|z|^2)^{\alpha}|f(z)|<\infty.
$$
The little weighted Herz space $\mathcal{H}^{0}_{\alpha}$ is the subspace of $\mathcal{H}^\infty_{\alpha}$ consisting of all $f \in \mathcal{H}^\infty_{\alpha}$ for which
$$
\lim_{|z|\rightarrow 1} (1-|z|^2)^{\alpha}|f(z)|=0.
$$
Note that when $\alpha<0$, the only function in $\mathcal{H}^\infty_{\alpha}$ (and hence in $\mathcal{H}^{0}_{\alpha}$) is the zero function.

The following theorem gives a complete characterization of the boundedness and compactness of $T_g$ from the Besov space $B_t(\B)$ to $F(p,\,q,\,s)$ in terms of a Bloch-type condition on the symbol $g$.

\begin{theorem} \label{TH5}
Assume $1\leq t<\infty$, $0<p<\infty$, $-n-1<q<\infty$, $1< s<\infty$, and $s+q>-1$. Let $g\in H(\B)$. Then
\begin{enumerate}
\item[(1)]  $T_g: B_t(\B)\rightarrow F(p,\,q,\,s)$ is bounded if and only if $g\in \mathcal{B}_{\frac{n+1+q}{p},\,\frac{t-1}{t}}$, in which case
\begin{equation}\label{norm-2}
\|T_g\|_{B_{t}(\B)\rightarrow F(p,\,q,\,s)}
\asymp\|g\|_{\mathcal{B}_{\frac{n+1+q}{p},\,\frac{t-1}{t}}}.
\end{equation}

\item[(2)]
 $T_g: B_t(\B)\rightarrow  F(p,\,q,\,s)$ is compact if and only if $g\in \mathcal{B}^0_{\frac{n+1+q}{p},\,\frac{t-1}{t}}$.
\end{enumerate}
\end{theorem}

\begin{proof}
For $g, f\in H(\B)$, it is easy to check that $T_gf(0)=0$ and
$$
\Re (T_g f)(z)=f(z) \Re g(z).
$$
Using the norm estimate \eqref{car-yub} for the space $F(p,\,q,\,s)$, we obtain, for any $f\in B_t(\B)$ and fixed $R>0$,
\begin{equation}\label{norm-1}
\|T_g f\|^p_{F(p,\,q,\,s)}\asymp \sup_{a\in \B}\frac{\int_{D(a, R)}|f(z)|^p |\Re g(z)|^p \, dv(z)}{(1-|a|^2)^{-q}}.
\end{equation}
Notice that
$$
d\mu_g(z)=|\Re g(z)|^p \, dv(z)
$$
is a positive Borel measure on $\B$, and so
$T_g$ is bounded from $B_t(\B)$ to $F(p,\,q,\,s)$ if and only if the embedding operator
$i: B_t(\B) \rightarrow T^{p}_{-q,\, 0}(\B, \mu_g)$ is bounded.

Suppose that $T_g: B_t(\B) \to F(p,\,q,\,s)$ is bounded. Theorem \ref{TH2} gives
\begin{align*}
\|T_g\|^p_{B_t(\B)\rightarrow F(p,\,q,\,s)}
&\asymp\|i\|^p_{B_t(\B)\rightarrow T^{p}_{-q,\, 0}(\B, \mu_g)}\\
&\asymp\sup_{a\in\B}\frac{\int_{D(a, R)}|\Re g(z)|^p \, dv(z)}{(1-|a|^2)^{-q}\left(\log\frac{2}{1-| a|^2}\right)^{\frac{p(1-t)}{t}}}<\infty.
\end{align*}
Note that
$$
|\Re g(a)|^p(1-| a|^2)^{n+1} \lesssim \int_{D(a, R)}|\Re g(z)|^p \, dv(z),
$$
and hence
$$
\sup_{a\in\B}|\Re g(a)|(1-| a|^2)^{\frac{n+1+q}{p}}\left(\log\frac{2}{1-| a|^2}\right)^{\frac{t-1}{t}} \lesssim\|T_g\|_{B_t(\B)\rightarrow F(p,\,q,\,s)}.
$$

Conversely, suppose that $g\in \mathcal{B}_{\frac{n+1+q}{p},\,\frac{t-1}{t}}$. For any $f\in B_t(\B)$, \eqref{norm-1} implies
\begin{align*}
\|T_g f\| ^p_{F(p,\,q,\,s)}
&\lesssim \|g\|^p_{\mathcal{B}_{\frac{n+1+q}{p},\,\frac{t-1}{t}}}
\sup_{a\in \B}\frac{\int_{D(a, R)}|f(z)|^p \, dv(z)}{(1-|a|^2)^{n+1}\left(\log\frac{2}{1-
|a|^2}\right)^{\frac{p(t-1)}{t}}}.
\end{align*}
By Theorem \ref{TH2}, the embedding operator
$i: B_t(\B)\rightarrow T_{n+1,\,\frac{p(t-1)}{t}}^{p}(\B, v)$ is bounded. Therefore,
\begin{align*}
\|T_g f\|^p_{F(p,\,q,\,s)}
&\lesssim \|g\|^p_{\mathcal{B}_{\frac{n+1+q}{p},\,\frac{t-1}{t}}}\|i\|^p_{B_t(\B)\rightarrow T_{n+1,\,\frac{p(t-1)}{t}}^{p}(\B, v)}\|f\|^p_{B_t}\\
&\lesssim \|g\|^p_{\mathcal{B}_{\frac{n+1+q}{p},\,\frac{t-1}{t}}}\|f\|^p_{B_t},
\end{align*}
and hence \eqref{norm-2} holds.

The compactness assertion in (2) follows similarly from Theorem \ref{TH3} and~\eqref{norm-1}.
\end{proof}

We now estimate the essential norms of the integral operators. Let $X$ and $Y$ be two Banach or quasi-Banach spaces, and let $T:X\to Y$ be a linear operator with operator norm $\|T\|_{X\to Y}$. The essential norm of $T$, denoted by $\|T\|_{e,X\to Y}$, is defined by
$$
\|T\|_{e,X\to Y}=\inf\{\|T-K\|_{X\to Y}: K \textrm{ is compact from } X \textrm{ to } Y\}.
$$
Clearly, $T$ is compact from $X$ to $Y$ if and only if $\|T\|_{e,X\to Y}=0$. When $X\subset Y$, the distance from a function $f\in Y$ to $X$ is defined by
$$
\dist_Y(f,X)=\inf\{\|f-g\|_Y:\ g\in X\}.
$$

\begin{theorem}\label{TH06}
Assume $1\leq t<\infty$, $0<p<\infty$, $-n-1<q<\infty$, $1< s<\infty$, and $s+q>-1$. Let $g\in \mathcal{B}_{\frac{n+1+q}{p},\,\frac{t-1}{t}}$. Then
\begin{align*}
&\|T_g\|_{e,\, B_{t}(\B)\rightarrow F(p,\,q,\,s)}\\
&\asymp \dist_{\mathcal{B}_{\frac{n+1+q}{p},\,\frac{t-1}{t}}}\left(g,\, \mathcal{B}^0_{\frac{n+1+q}{p},\,\frac{t-1}{t}}\right)\\
&\asymp\limsup_{|z|\rightarrow 1}|\Re g(z)|(1-| z|^2)^{\frac{n+1+q}{p}}\left(\log\frac{2}{1-| z|^2}\right)^{\frac{t-1}{t}}.
\end{align*}
\end{theorem}

\begin{proof}
Since $g\in \mathcal{B}_{\frac{n+1+q}{p},\,\frac{t-1}{t}}$, by Theorem \ref{TH5}, we know that $T_g: B_t(\B)\rightarrow F(p,\,q,\,s)$ is bounded.
Given $a\in\B$, set
\begin{equation}\label{textfun-1}
f_a(z)=\left(\log\frac{2}{1-|a|^2}\right)^{-1/t}\log\frac{2}{1-\langle z,a \rangle},\qquad z\in\B.
\end{equation}
As in the proof of Theorem \ref{TH2}, we have $f_a\in B_t(\B)$ with $\|f_a\|_{B_t}\lesssim 1$, and $\{f_a\}$ converges to $0$ uniformly on compact subsets of $\B$ as $|a|\rightarrow 1$.
Hence, for each compact operator $K$ from $B_t(\B)$ to $F(p,\,q,\,s)$, there exists some constant $C>0$ such that
\begin{align*}
&\|T_g-K\|^p_{B_{t}(\B)\rightarrow F(p,\,q,\,s)}
\geq C\limsup_{|a|\rightarrow 1}\|(T_g-K) f_a\|^p_{F(p,\,q,\,s)}\\
&\geq C\left(\limsup_{|a|\rightarrow 1} \|T_g (f_a)\|^p_{F(p,\,q,\,s)}-\lim_{|a|\rightarrow 1} \|K(f_a)\|^p_{F(p,\,q,\,s)}\right)\\
&= C\limsup_{|a|\rightarrow 1} \|T_g (f_a)\|^p_{F(p,\,q,\,s)}.
\end{align*}
By \eqref{norm-1},
\begin{align*}
\|T_g (f_a)\|^p_{F(p,\,q,\,s)} &\asymp \sup_{\xi\in \B}\frac{\int_{D(\xi,R)}|f_a(w)|^p |\Re g(w)|^p dv(w)}{(1-|\xi|^2)^{-q}}\\
&\geq \frac{\int_{D(a, R)}|f_a(w)|^p |\Re g(w)|^p dv(w)}{(1-|a|^2)^{-q}}\\
&\geq C(1-|a|^2)^{n+1+q}|f_a(a)|^p|\Re g(a)|^p\\
&=C|\Re g(a)|^p(1-|a|^2)^{n+1+q}\left(\log\frac{2}{1-|a|^2}\right)^{\frac{p(t-1)}{t}}.
\end{align*}
Thus,
$$
\|T_g-K\|_{B_{t}(\B)\rightarrow F(p,\,q,\,s)}
\gtrsim
\limsup_{|a|\rightarrow 1}|\Re g(a)|(1-|a|^2)^{\frac{n+1+q}{p}}
\left(\log\frac{2}{1-|a|^2}\right)^{\frac{t-1}{t}}.
$$

We now prove the reverse inequality. For $r\in (0,1)$, define
$$
g_r(z)=g(rz), \quad z\in \B.
$$
As in the proof of \cite[Theorem 3.9]{Tj06}, we can show that $g_r\in \mathcal{B}^0_{\frac{n+1+q}{p},\,\frac{t-1}{t}}$. Moreover,
\begin{align*}
& \operatorname{dist}_{\mathcal{B}_{\frac{n+1+q}{p},\,\frac{t-1}{t}}}\left(g,\, \mathcal{B}^0_{\frac{n+1+q}{p},\,\frac{t-1}{t}}\right)\\
&\asymp \limsup_{r\rightarrow 1}\|g-g_r\|_{\mathcal{B}_{\frac{n+1+q}{p},\,\frac{t-1}{t}}}\\
&\asymp \limsup_{|a|\rightarrow 1}|\Re g(a)|(1-|a|^2)^{\frac{n+1+q}{p}}\left(\log\frac{2}{1-|a|^2}\right)^{\frac{t-1}{t}}.
\end{align*}
Theorem \ref{TH5} yields that $T_{g_r}: B_t(\B)\rightarrow F(p,\,q,\,s)$ is compact, and
\begin{align*}
\|T_g\|_{e,\,B_{t}(\B)\rightarrow F(p,\,q,\,s)}
&\leq \|T_g-T_{g_r}\|_{B_{t}(\B)\rightarrow F(p,\,q,\,s)}\\
&\asymp \|g-g_r\|_{\mathcal{B}_{\frac{n+1+q}{p},\,\frac{t-1}{t}}}.
\end{align*}
Taking $r\rightarrow 1$, we have
\begin{align*}
\|T_g\|_{e,\,B_{t}(\B)\rightarrow F(p,\,q,\,s)}
&\lesssim \limsup_{r\rightarrow 1}\|g-g_r\|_{\mathcal{B}_{\frac{n+1+q}{p},\,\frac{t-1}{t}}}\\
&\asymp \operatorname{dist}_{\mathcal{B}_{\frac{n+1+q}{p},\,\frac{t-1}{t}}}\left(g,\, \mathcal{B}^0_{\frac{n+1+q}{p},\,\frac{t-1}{t}}\right).
\end{align*}
\end{proof}

Finally, we study the boundedness, compactness, and essential norm of $I_g$ from the Besov space $B_t(\B)$ to $F(p,\,q,\,s)$.

\begin{theorem}\label{thm6}
Let $g\in H(\B)$, $n<t<\infty$, $0<p<\infty$, $-n-1<q<\infty$, $1<s<\infty$, and $s+q>-1$. Then
\begin{enumerate}
\item[(1)] $I_g: B_t(\B)\rightarrow F(p,\,q,\,s)$ is bounded if and only if $g\in \mathcal{H}^\infty_{\frac{n+1+q-p}{p}}$, in which case
\begin{equation}\label{norm-a2}
\|I_g\|_{B_{t}(\B)\rightarrow F(p,\,q,\,s)}
\asymp\|g\|_{\mathcal{H}^\infty_{\frac{n+1+q-p}{p}}}.
\end{equation}

\item[(2)] $I_g: B_t(\B)\rightarrow F(p,\,q,\,s)$ is compact if and only if
$g\in \mathcal{H}^0_{\frac{n+1+q-p}{p}}$.

\item[(3)] If $I_g: B_t(\B)\rightarrow F(p,\,q,\,s)$ is bounded, then
\begin{equation}\label{norm-b4}
\|I_g\|_{e,\, B_{t}(\B)\rightarrow F(p,\,q,\,s)}
\asymp\limsup_{|z|\rightarrow 1}|g(z)|(1-| z|^2)^{\frac{n+1+q-p}{p}}.
\end{equation}
\end{enumerate}
\end{theorem}

\begin{proof}
For $g,f\in H(\B)$, it is easy to check that $I_gf(0)=0$ and
$$
\Re (I_g f)(z)=g(z)\Re f(z).
$$
For every $f\in B_t(\B)$, \eqref{car-yub} implies
\begin{equation}\label{norm-04}
\|I_g f\|^p_{F(p,\,q,\,s)}\asymp \sup_{a\in \B}\frac{\int_{D(a, R)}|g(z)|^p |\Re f(z)|^p \, dv(z)}{(1-|a|^2)^{-q}}
\end{equation}
for some (or all) $R>0$.

Note that, when $n<t<\infty$, it follows from \cite[page 231, Exercise 6.8]{Zh04} that
$f\in B_t(\B)$ if and only if $\Re f \in A^t_{t-n-1}(\B)$, and
$$
\|f\|_{B_t}\asymp |f(0)|+\|\Re f\|_{A^t_{t-n-1}}.
$$
Thus, $I_g$ is bounded from $B_t(\B)$ to $F(p,\,q,\,s)$ if and only if the embedding operator
$$
i: A^t_{t-n-1}(\B) \rightarrow T^{p}_{-q,\,0}(\B, \nu_g)
$$
is bounded, where
$$
d\nu_g(z)=|g(z)|^p\,dv(z).
$$
By Theorem \ref{TH4}, this holds if and only if
\begin{align*}
\|I_g\|^p_{B_t(\B)\rightarrow F(p,\,q,\,s)}
&\asymp\|i\|^p_{A^t_{t-n-1}(\B)\rightarrow T^{p}_{-q,\,0}(\B, \nu_g)}\\
&\asymp\sup_{a\in\B}\frac{\int_{D(a, R)}|g(z)|^p \, dv(z)}{(1-|a|^2)^{p-q}}<\infty.
\end{align*}
Thus,
$$
\sup_{a\in\B}|g(a)|(1-|a|^2)^{\frac{n+1+q-p}{p}}
\lesssim \|I_g\|_{B_t(\B)\rightarrow F(p,\,q,\,s)}<\infty.
$$

Conversely, suppose that $g\in \mathcal{H}^\infty_{\frac{n+1+q-p}{p}}$. By \eqref{norm-04},
\begin{align*}
\|I_g f\| ^p_{F(p,\,q,\,s)}
&\lesssim \|g\|^p_{\mathcal{H}^\infty_{\frac{n+1+q-p}{p}}}
 \sup_{a\in \B}\frac{\int_{D(a, R)}|\Re f(z)|^p \, dv(z)}{(1-| a|^2)^{n+1-p}}
\end{align*}
for any $f\in B_t(\B)$. Applying Theorem \ref{TH4} again, we see that the embedding
$$
i: A^t_{t-n-1}(\B)\rightarrow T_{n+1-p,\,0}^{p}(v)
$$
is bounded. Therefore,
\begin{align*}
\|I_g f\|^p_{F(p,\,q,\,s)}
&\lesssim \|g\|^p_{\mathcal{H}^\infty_{\frac{n+1+q-p}{p}}}
\|i\|^p_{A^t_{t-n-1}(\B)\rightarrow T_{n+1-p,\,0}^{p}(v)}
\|\Re f\|^p_{A^t_{t-n-1}}\\
&\lesssim \|g\|^p_{\mathcal{H}^\infty_{\frac{n+1+q-p}{p}}}\|f\|^p_{B_t}.
\end{align*}
Therefore, \eqref{norm-a2} follows.

Similarly, Theorem \ref{TH4} and \eqref{norm-04} imply that
$I_g: B_t(\B)\rightarrow F(p,\,q,\,s)$ is compact if and only if
$g\in \mathcal{H}^0_{\frac{n+1+q-p}{p}}$.
It remains to prove \eqref{norm-b4}.

For $a\in \B$, define
$$
f_a(z)=\frac{(1-|a|^2)^{\frac{n+1}{t}}}{(1-\langle z, a\rangle)^{\frac{n+1}{t}}},\quad z\in\B.
$$
Then
$$
\Re f_a(z)=\frac{n+1}{t}\frac{(1-|a|^2)^{\frac{n+1}{t}}}{(1-\langle z, a\rangle)^{\frac{n+1}{t}+1}}\langle z, a\rangle,
$$
so
$$
\|f_a\|_{B_t}\asymp 1+(1-|a|^2)^{\frac{n+1}{t}}
\left(\int_{\B}\frac{(1-|z|^2)^{t-n-1}}{|1-\langle z, a\rangle|^{n+1+t}}\,dv(z)\right)^{1/t}\lesssim 1,
$$
and $\{f_a\}$ converges to $0$ uniformly on compact subsets of $\B$ as $|a|\rightarrow 1$.

For each compact operator $K$ from $B_{t}(\B)$ to $F(p,\,q,\,s)$, we have
\begin{align*}
&\|I_g-K\|^p_{B_{t}(\B)\rightarrow F(p,\,q,\,s)}
\geq C\limsup_{|a|\rightarrow 1}\|(I_g-K) f_a\|^p_{F(p,\,q,\,s)}\\
&\geq C\left(\limsup_{|a|\rightarrow 1}\|I_g(f_a)\|^p_{F(p,\,q,\,s)}
-\lim_{|a|\rightarrow 1}\|K(f_a)\|^p_{F(p,\,q,\,s)}\right)\\
&= C\limsup_{|a|\rightarrow 1}\|I_g(f_a)\|^p_{F(p,\,q,\,s)}.
\end{align*}
By \eqref{norm-04}, for each fixed $a$,
\begin{align*}
\|I_g (f_a)\|^p_{F(p,\,q,\,s)}
&\asymp \sup_{\xi\in \B}\frac{\int_{D(\xi,R)}|\Re f_a(w)|^p |g(w)|^p \, dv(w)}{(1-|\xi|^2)^{-q}}\\
&\geq \frac{\int_{D(a,R)}|\Re f_a(w)|^p |g(w)|^p \, dv(w)}{(1-|a|^2)^{-q}}\\
&\gtrsim |g(a)|^p(1-|a|^2)^{n+1+q-p}.
\end{align*}
Thus,
$$
	\|I_g-K\|_{B_{t}(\B)\rightarrow F(p,\,q,\,s)}
	\gtrsim
	\limsup_{|a|\rightarrow 1}|g(a)|(1-|a|^2)^{\frac{n+1+q-p}{p}}.
$$

For the reverse inequality, set $g_r(z)=g(rz)$ for $z\in\B$. Then $g_r\in \mathcal{H}^0_{\frac{n+1+q-p}{p}}$, and
\begin{align*}
& \operatorname{dist}_{\mathcal{H}^\infty_{\frac{n+1+q-p}{p}}}
\left(g,\, \mathcal{H}^0_{\frac{n+1+q-p}{p}}\right)
\asymp \limsup_{r\rightarrow 1}\|g-g_r\|_{\mathcal{H}^\infty_{\frac{n+1+q-p}{p}}}\\
&\asymp \limsup_{|a|\rightarrow 1}|g(a)|(1-|a|^2)^{\frac{n+1+q-p}{p}}.
\end{align*}
Since $I_{g_r}: B_t(\B)\rightarrow F(p,\,q,\,s)$ is compact, we get
\begin{align*}
\|I_g\|_{e,\,B_{t}(\B)\rightarrow F(p,\,q,\,s)}
&\leq \|I_g-I_{g_r}\|_{B_{t}(\B)\rightarrow F(p,\,q,\,s)}\\
&\asymp \|g-g_r\|_{\mathcal{H}^\infty_{\frac{n+1+q-p}{p}}}.
\end{align*}
Taking $r\rightarrow 1$, we have
\begin{align*}
\|I_g\|_{e,\,B_{t}(\B)\rightarrow F(p,\,q,\,s)}
&\lesssim \limsup_{r\rightarrow 1}\|g-g_r\|_{\mathcal{H}^\infty_{\frac{n+1+q-p}{p}}}\\
&\asymp \operatorname{dist}_{\mathcal{H}^\infty_{\frac{n+1+q-p}{p}}}
\left(g,\, \mathcal{H}^0_{\frac{n+1+q-p}{p}}\right).
\end{align*}
\end{proof}

\begin{remark}
From Theorems \ref{TH5}--\ref{thm6}, we see that the boundedness, compactness, and essential norms of the integral operators $T_g$ and $I_g$ from the Besov space $B_t(\B)$ to $F(p,q,s)$ are independent of the parameter $s$.
\end{remark}

\section*{Acknowledgments}

The third author thanks Huzhou Normal University for its hospitality during a visit in January 2026 when part of this research was carried out.

\end{document}